\documentclass{amsart}

\usepackage{mathpazo}
\usepackage{algorithm}
\usepackage{algpseudocode}
\usepackage{amsmath,amssymb,amscd,color,amsthm}
\usepackage{mathrsfs}
\usepackage{mathtools}
\usepackage{graphicx}
\usepackage{float}
\usepackage{bm}
\usepackage{comment}
\definecolor{green}{RGB}{0,144,0}
\definecolor{bluegreen}{RGB}{17,100,180}

\usepackage[linkcolor = bluegreen, citecolor = bluegreen, colorlinks = True]{hyperref}
\usepackage[all]{hypcap}

\newtheorem{theorem}{Theorem}[section]
\newtheorem*{theorem*}{Theorem}

\newtheorem{lemma}[theorem]{Lemma}
\newtheorem{corollary}[theorem]{Corollary}
\newtheorem{proposition}[theorem]{Proposition}

\newtheorem{remark}[theorem]{Remark}

\newcommand{\N}{\mathbb{N}}
\newcommand{\Z}{\mathbb{Z}}
\newcommand{\Q}{\mathbb{Q}}
\newcommand{\R}{\mathbb{R}}

\DeclareMathOperator{\HD}{\text{HD}}

\title{New gaps on the Lagrange and Markov spectra}

\author{Luke Jeffreys}
\address{School of Mathematics , University of Bristol, Fry Building, Woodland Road, Bristol BS8 1UG}
\curraddr{}
\email{luke.jeffreys@bristol.ac.uk}
\thanks{The first author is thankful for support from the Heilbronn Institute for Mathematical Research}

\author{Carlos Matheus}
\address{Centre de Math{\'e}matiques Laurent Schwartz, {\'E}cole Polytechnique, 91128 Palaiseau Cedex, France}
\curraddr{}
\email{carlos.matheus@math.cnrs.fr}

\author{Carlos Gustavo Moreira}
\address{School of Mathematical Sciences, Nankai University, Tianjin 300071, P. R. China, and IMPA, Estrada Dona Castorina 110, CEP 22460-320, Rio de Janeiro, Brazil}
\curraddr{}
\email{gugu@impa.br}

\date{}


\begin{document}

\begin{abstract}
Let $L$ and $M$ denote the Lagrange and Markov spectra, respectively. It is known that $L\subset M$ and that $M\setminus L\neq\varnothing$. In this work, we exhibit new gaps of $L$ and $M$ using two methods. First, we derive such gaps by describing a new portion of $M\setminus L$ near to 3.938: this region (together with three other candidates) was found by investigating the pictures of $L$ recently produced by V. Delecroix and the last two authors with the aid of an algorithm explained in one of the appendices to this paper. As a by-product, we also get the largest known elements of $M\setminus L$ and we improve upon a lower bound on the Hausdorff dimension of $M\setminus L$ obtained by the last two authors together with M. Pollicott and P. Vytnova (heuristically, we get a new lower bound of $0.593$ on the dimension of $M\setminus L$). Secondly, we use a renormalisation idea and a thickness criterion (reminiscent from the third author's PhD thesis) to detect infinitely many maximal gaps of $M$ accumulating to Freiman's gap preceding the so-called Hall's ray  $[4.52782956616...,\infty)\subset L$. 
\end{abstract}

\maketitle


\section{Introduction}

The classical theory of Diophantine approximation is concerned with how well irrational numbers can be approximated by rational numbers. Given a positive real number $\alpha$ we define its \emph{best constant of Diophantine approximation} to be
\[L(\alpha) := \limsup_{p,q\to\infty} \frac{1}{|q(q\alpha - p)|}.\]
In a sense, $L(\alpha)$ is the largest constant so that the inequality
\[\left|\alpha - \frac{p}{q}\right| < \frac{1}{L(\alpha)q^{2}}\]
has infinitely many solutions $p,q \in \N, q\neq 0$. The \emph{Lagrange spectrum} is defined to be the set
\[L:=\{L(\alpha)\,\mid\,\alpha\in\R\setminus\Q\}.\]

Perron~\cite{Pe21} proved that if we have the continued fraction expansion
\[\alpha = [a_{0};a_{1},a_{2},\ldots] := a_{0} + \frac{1}{a_{1} + \frac{1}{a_{2} + \frac{1}{\dots}}},\]
then we have
\[L(\alpha) = \limsup_{n\to\infty}\,( [a_{n};a_{n-1},\ldots,a_{0}] + [0;a_{n+1},a_{n+2},\ldots]).\]
As such, we are also able to define the Lagrange spectrum in terms of the bi-infinite shift space $\Sigma := \{1,2,3,\ldots\}^{\Z}$. More specifically, for $(a_{i})_{i\in\Z}\in\Sigma$ we define
\[\lambda_{0}((a_{i})_{i\in\Z}) := [a_{0};a_{-1},a_{-2},\ldots] + [0;a_{1},a_{2},\ldots],\]
and, for $j\in\Z$,
\[\lambda_{j}((a_{i})_{i\in\Z}) := \lambda_{0}(\sigma^{j}((a_{i})_{i\in\Z})) = \lambda_{0}((a_{i+j})_{i\in\Z}),\]
where $\sigma:\Sigma\to\Sigma$ is the left-shift sending $(a_{i})_{i\in\Z}$ to $(a_{i+1})_{i\in\Z}$. We can now define the Lagrange spectrum to be
\[L:= \{\limsup_{j\to\infty}\lambda_{j}(\underline{a})\,\mid\,\underline{a}\in\Sigma\}.\]
Similarly, given $(a_{i})_{i\in\Z}\in\Sigma$ we define
\[m((a_{i})_{i\in\Z}) := \sup_{n\in\Z}\lambda_{n}((a_{i})_{i\in\Z}).\]
Then the \emph{Markov spectrum} is defined to be the set
\[M:= \{m(\underline{a}) \,\mid\, \underline{a}\in\Sigma\}.\]
In the sequel, we will write a sequence $(a_{i})_{i\in\Z}$ as the string $\ldots a_{-2}a_{-1}a_{0}^{*}a_{1}a_{2}\ldots$ where the asterisk denotes the 0th position. We will also use an overline to denote periodicity so that, for example, the sequence $a_{i} = (i\mod 3) + 1$ is denoted $\overline{1^{*}23} = \ldots1231231^{*}23123123\ldots$. This notation should be clear from the context as we will mostly restrict to the subshift $\{1,2,3,4\}^{\Z}$ so, in particular, all $a_{i}$ will be single digits.

Markov~\cite{Ma79,Ma80} first studied the spectra $L$ and $M$ around 1880. It is known that $L\subset M\subset\R^{+}$ with $L\cap(0,3) = M\cap(0,3)$ an explicit discrete set. In 1975, Freiman~\cite{Fr75} showed that $[\mu,\infty)\subset L\subset M$, and $(\nu,\mu)\cap M = \varnothing$ with $\nu,\mu\in M$, where
\[\nu = \lambda_{0}(\overline{323444}313134^{*}313121133\overline{313121}) = 4.52782953841\ldots\]
and
\[\mu = \lambda_{0}(\overline{121313}22344^{*}3211\overline{313121}) = 4.52782956616\ldots.\]
The ray $[\mu,\infty)$ is known as Hall's ray after earlier work of Hall~\cite{Ha47} (see also the intermediate results of Freiman-Judin~\cite{FJ66}, Hall~\cite{Ha71}, Freiman~\cite{Fr73} and Schecker~\cite{Sc77}).

Freiman~\cite{Fr68} also showed that $M\setminus L \neq \varnothing$. In fact, the second and third authors together with M. Pollicott and P. Vytnova~\cite{MMPV22} recently proved that the Hausdorff dimension $\HD(M\setminus L)$ of $M\setminus L$ satisfies
\[0.537152 < \HD(M\setminus L) <  0.796445.\]
We direct the reader to the survey~\cite{MM21} and the textbooks of Cusick-Flahive~\cite{CF89} and Lima-Matheus-Moreira-Roma{\~n}a~\cite{L+20} for more details on these spectra.

\subsection{A new portion of \boldmath{$M\setminus L$}}

Our first result finds a new portion of $M\setminus L$ and gives an improved lower bound for its Hausdorff dimension.

\begin{theorem}\label{t:main1}
The intersection of $M\setminus L$ with the interval $(3.938, 3.939)$ is non-empty. The largest known element of $M\setminus L$ is
\[m(\overline{12}331113311321231133311121211333^{*}\overline{11121211333})=3.938776241989784909...\,\,.\]
\end{theorem}

\begin{remark} Our proof of this result yields that the local dimension of $M\setminus L$ near 3.938 coincides with the dimension of a dynamically defined Cantor set which is richer than the Cantor set $\Omega$ considered in \cite[\S 4.6.5]{MMPV22}. In particular, this improves the lower bound on $\HD(M\setminus L)$ and, in fact, a heuristic computation (based on the so-called Jenkinson--Pollicott method) indicates that $\HD(M\setminus L)>0.593$: see the next section. 
\end{remark}

The proof of this result is contained in Section~\ref{s:M-L}. We also, in Appendix~\ref{app:M-L}, give some additional newly discovered portions of $M\setminus L$. We do not give the proof of these claims as they do not lead to significantly better estimates of the Hausdorff dimension of $M\setminus L$.

\subsection{New maximal gaps of \boldmath{$M$}}

Our second result concerns maximal gaps in the Markov spectrum $M$. Recall that Freiman proved that the gap $(\nu,\mu)$ is a maximal gap of $M$. We find infinitely many new maximal gaps of $M$ accumulating to Freiman's gap. Specifically, we prove the following.

\begin{theorem}\label{t:main2} There is a sequence $(\alpha_n,\beta_n)$ of maximal gaps of $M$ such that $\lim\limits_{n\to\infty}\beta_n = \nu$. 
\end{theorem}

In Section~\ref{s:Frei-gap}, we give a proof of Freiman's result that $(\nu,\mu)$ is a maximal gap since the contributing lemmas are used in Section~\ref{s:accum-gaps} in which we prove Theorem~\ref{t:main2} via a renormalisation idea (leading to a sort of ``recurrence on scales'') and a thickness criterion in the spirit of the discussion of \cite{Mo96}.

\subsection{Computational assistance in the investigations of \boldmath{$M\setminus L$}}

The candidate sequence giving rise to elements of $M\setminus L$ analysed in Section~\ref{s:M-L} and those discussed in the appendix were discovered with the assistance of a computer search. The code was essentially running the arguments we will give in Section~\ref{s:M-L} which are themselves similar to those given in previous work of the second and third authors concerning elements of  $M\setminus L$ near to $3.7096$~\cite{MM20}.

We now describe the ideas behind the computer search. Firstly, for a candidate finite sequence $a$ we determine the Markov value of the periodic sequence $s = \overline{a}$ determined by $a$. We then consider modifications of this sequence $s$ where we force the sequence to instead terminate by $\overline{21}$ to the right or by $\overline{12}$ to the left. We find the modification that gives the smallest increase in the corresponding Markov value. Call this modified sequence $w$. Next, we try to determine the central portions of sequences that could give rise to Markov values in the range $[m(s),m(w)+\epsilon]$, for some small (possibly negative) $\epsilon$. By searching for central portions of larger and larger length we can observe evidence for the one-sided periodicity we hope to make use of in the arguments given in Section~\ref{s:M-L}. If we see no evidence for such one-sided periodicity after searching for central portions of a reasonable length then we throw out the candidate $a$ and try for a new finite sequence. The pseudo-code describing the algorithm used to determine the central portions of candidate sequences is given in Appendix~\ref{app:alg}.

In practice the candidate finite sequences $a$ are chosen to be odd length non-semi-symmetric words, where a word is semi-symmetric if it is a palindrome or a concatenation of two palindromes. We direct the reader to~\cite[Subsection 1.3]{MM20} for a discussion of why odd length non-semi-symmetric words are natural candidates for finding elements of $M\setminus L$.


\section{A new portion of $M\setminus L$ near $3.938$}\label{s:M-L}

We consider the word of odd length $11121211333$. Note that it is non-semi-symmetric (in the sense of Flahive), i.e., it is not a palindrome nor a concatenation of two palindromes. 

The Markov value of the associated periodic sequence is  
$$\lambda_0(\overline{11121211333^*}) = 3.93877624198\textcolor{red}{10}28026\dots$$ 

Generally speaking, our goal below is to show that a portion of $M\setminus L$ occurs near 
$$\lambda_0(\overline{12}12121133311121211333^*\overline{11121211333}) = 3.93877624198\textcolor{red}{11}39302\dots$$

In the sequel, we shall study a sequence $(\dots, x_{-m}, \dots, x_{-1}, x_0^*, x_1, \dots, x_n, \dots)\in\{1,2,3\}^{\mathbb{Z}}$ with a Markov value $m(x)=\lambda_0(x)$ nearby $3.93877624198\textcolor{red}{11}$. 

For a finite sequence $a$, inequalities of the form $\lambda_{0}(a)>v$, say, mean that we have $\lambda_{0}(w)>v$ for all bi-infinite sequence $w$ that are obtained by extending $a$ on both sides.

\subsection{Local uniqueness} 

Note that $x_0=3$. Moreover, the possible vicinities of $x_0^*$ (up to transposition) are $13^*1$, $13^*2$, $13^*3$, $23^*2$, $23^*3$, $33^*3$. 

\begin{lemma}\phantomsection\label{l:M-L1}
\begin{itemize}
\item[(i)] $\lambda_0(13^*1)>4.11$
\item[(ii)] $\lambda_0(33^*3)\leq\lambda_0(33^*2)\leq\lambda_0(23^*2)<3.884$
\end{itemize}
\end{lemma} 

By the previous lemma, up to transposition, it suffices to analyse the extensions to the right of $23^*1$ and $33^*1$, i.e., $23^*11$, $23^*12$, $23^*13$, $33^*11$, $33^*12$, $33^*13$. 

\begin{lemma}\label{l:M-L2} $\lambda_0(3^*13)>\lambda_0(3^*12)>3.957$. 
\end{lemma} 

By the previous lemma, it suffices to analyse the extensions to the left of $23^*11$ and $33^*11$, i.e., $123^*11$, $223^*11$, $323^*11$, $133^*11$, $233^*11$, $333^*11$. 

\begin{lemma}\phantomsection\label{l:M-L3}
\begin{itemize}
\item[(i)] $\lambda_0(323^*11)>\lambda_0(223^*11)>3.9678$
\item[(ii)] $\lambda_0(133^*11)<3.9228$ 
\end{itemize}
\end{lemma}

By the previous lemma, it suffices to analyse the extensions to the right of $123^*11$, $233^*11$, $333^*11$, i.e., $123^*111$, $123^*112$, $123^*113$, $233^*111$, $233^*112$, $233^*113$, $333^*111$, $333^*112$, $333^*113$. 

\begin{lemma}\phantomsection\label{l:M-L4} 
\begin{itemize}
\item[(i)] $\lambda_0(123^*111) > 3.9673$
\item[(ii)] if $131$ and $312$ are forbidden, then $\lambda_0(233^*113)<\lambda_0(233^*112)<\lambda_0(233^*111)\leq \lambda_0(21233^*11132) < 3.93676$ 
\item[(iii)] $\lambda_0(333^*113)<\lambda_0(333^*112) < 3.8969$
\end{itemize}
\end{lemma} 

By the previous lemma, it suffices to analyse the extensions to the left of $123^*112$, $123^*113$, $333^*111$, i.e., $1123^*112$, $2123^*112$, $3123^*112$, $1123^*113$, $2123^*113$, $3123^*113$, $1333^*111$, $2333^*111$, $3333^*111$. 

\begin{lemma}\phantomsection\label{l:M-L5} 
\begin{itemize}
\item[(i)] $\lambda_0(1123^*112)>\lambda_0(2123^*112)>3.9414$; in particular, $123^*112$ is forbidden if $312$ is forbidden 
\item[(ii)] $\lambda_0(2123^*113) < 3.93768$ 
\item[(iii)] if $131$ is forbidden, then $\lambda_0(1123^*113)\geq \lambda_0(1123^*11323) > 3.9419$  
\end{itemize}
\end{lemma} 

By the previous lemma, it suffices to analyse the extensions to the right of $1333^*111$, $2333^*111$, $3333^*111$, i.e., $1333^*1111$, $1333^*1112$, $1333^*1113$, $2333^*1111$, $2333^*1112$, $2333^*1113$, $3333^*1111$, $3333^*1112$, $3333^*1113$. 

\begin{lemma}\phantomsection\label{l:M-L6}
\begin{itemize}
\item[(i)] $\lambda_0(333^*1113)>3.94084$ 
\item[(ii)] $\lambda_0(3333^*1111)<\lambda_0(2333^*1111)<\lambda_0(1333^*1111)< 3.92786$ 
\item[(iii)] $\lambda_0(3333^*1112)<\lambda_0(2333^*1112)< 3.93844$ 
\end{itemize}
\end{lemma}  

By the previous lemma, it suffices to analyse the extensions to the left of $1333^*1112$, i.e., $11333^*1112$, $21333^*1112$, $31333^*1112$. Since $213$ and $313$ are forbidden (cf. Lemma \ref{l:M-L2}), our task is reduced to study the extensions to the right of $11333^*1112$, i.e., $11333^*11121$, $11333^*11122$, $11333^*11123$. 

\begin{lemma}\label{l:M-L8} $\lambda_0(11333^*11123)<\lambda_0(11333^*11122)<3.93631$
\end{lemma} 

By the previous lemma, it suffices to analyse the extensions to the left and right of $11333^*11121$ (while taking into account that $213$ is forbidden), i.e., $111333^*111211$, $211333^*111211$, $311333^*111211$, $111333^*111212$, $211333^*111212$, $311333^*111212$. 

\begin{lemma}\label{l:M-L10} $\lambda_0(311333^*111211)<\lambda_0(211333^*111211)<\lambda_0(111333^*111211)< 3.938464$
\end{lemma} 

By the previous lemma (and after recalling that $131$ and $3111333$ are forbidden, cf. Lemmas \ref{l:M-L1} and \ref{l:M-L6} (i)), it suffices to analyse the extensions to the left of  $111333^*111212$, $211333^*111212$, $311333^*111212$, i.e., $1111333^*111212$, $1211333^*111212$,  $2111333^*111212$, $2211333^*111212$, $2311333^*111212$, $3211333^*111212$, $3311333^*111212$. 

\begin{lemma}\label{l:M-L11} $\lambda_0(2111333^*111212) > 3.93889$
\end{lemma} 

By the previous lemma, it suffices to analyse the extensions to the right of $1111333^*111212$, $1211333^*111212$, $2211333^*111212$, $2311333^*111212$, $3211333^*111212$, $3311333^*111212$, i.e., 
\begin{itemize}
\item $1111333^*1112121$, $1111333^*1112122$, $1111333^*1112123$
\item $1211333^*1112121$, $1211333^*1112122$, $1211333^*1112123$ 
\item $2211333^*1112121$, $2211333^*1112122$, $2211333^*1112123$
\item $2311333^*1112121$, $2311333^*1112122$, $2311333^*1112123$
\item $3211333^*1112121$, $3211333^*1112122$, $3211333^*1112123$
\item $3311333^*1112121$, $3311333^*1112122$, $3311333^*1112123$
\end{itemize} 

\begin{lemma}\phantomsection\label{l:M-L12} 
\begin{itemize}
\item[(i)] $\lambda_0(1111333^*1112121)>\lambda_0(1111333^*1112122)> 3.938835$ 
\item[(ii)] $\max\{\lambda_0(1211333^*1112123), \lambda_0(1211333^*1112122), \lambda_0(2211333^*1112123)\}<\lambda_0(2211333^*1112122)<3.938751$
\item[(iii)] $\lambda_0(3211333^*1112121)>\lambda_0(2211333^*1112121)>3.938824$ 
\item[(iv)] $\lambda_0(3211333^*1112123)$, $\lambda_0(2311333^*1112122)$, $\lambda_0(2311333^*1112123)$, $\lambda_0(3311333^*1112122)$, $\lambda_0(3311333^*1112123)$ $<$ $\lambda_0(3211333^*1112122) < 3.9387718$
\end{itemize}
\end{lemma} 

By the previous lemma (and after recalling that $312$, $22311$ and $32311$ are forbidden, cf. Lemmas \ref{l:M-L2} and \ref{l:M-L3} (i)), it suffices to analyse the extensions to the left of $1111333^*1112123$, $1211333^*1112121$, $2311333^*1112121$, $3311333^*1112121$, i.e., 
\begin{itemize}
\item $11111333^*1112123$, $21111333^*1112123$, $31111333^*1112123$
\item $11211333^*1112121$, $21211333^*1112121$
\item $12311333^*1112121$ 
\item $13311333^*1112121$, $23311333^*1112121$, $33311333^*1112121$
\end{itemize} 

\begin{lemma}\phantomsection\label{l:M-L13} 
\begin{itemize}
\item[(i)] $\lambda_0(11111333^*1112123) > 3.9388049$ 
\item[(ii)] $\lambda_0(11211333^*1112121) > 3.9387855$ 
\end{itemize}
\end{lemma} 

By the previous lemma (and after recalling that $213$ is forbidden), it suffices to analyse the extensions to the right of  $21111333^*1112123$, $31111333^*1112123$, $21211333^*1112121$, $12311333^*1112121$, $13311333^*1112121$, $23311333^*1112121$, $33311333^*1112121$, i.e., 
\begin{itemize}
\item $21111333^*11121231$, $21111333^*11121232$, $21111333^*11121233$
\item $31111333^*11121231$, $31111333^*11121232$, $31111333^*11121233$
\item $21211333^*11121211$, $21211333^*11121212$ 
\item $12311333^*11121211$, $12311333^*11121212$
\item $13311333^*11121211$, $13311333^*11121212$ 
\item $23311333^*11121211$, $23311333^*11121212$ 
\item $33311333^*11121211$, $33311333^*11121212$  
\end{itemize} 

\begin{lemma}\phantomsection\label{l:M-L14} 
\begin{itemize}
\item[(i)] if $312$ and $313$ are forbidden, then $\lambda_0(21111333^*1112123)\geq \lambda_0(21111333^*111212311) > 3.93877973$
\item[(ii)] if $312$ and $313$ are forbidden, then $\lambda_0(31111333^*11121231)\leq \lambda_0(31111333^*111212311)< 3.938775326$ 
\item[(iii)] if $131$ is forbidden, then $\lambda_0(31111333^*11121233)>\lambda_0(31111333^*11121232)\geq \lambda_0(231111333^*11121232)>3.9387807$ 
\item[(iv)] $\lambda_0(21211333^*11121212)>\lambda_0(3311333^*11121212)>3.938783$ 
\item[(v)] $\lambda_0(12311333^*11121211)<\lambda_0(3311333^*11121211)<3.9387521$
\end{itemize}
\end{lemma} 

By the previous lemma (and after recalling that $312$ and $1123113$ are forbidden, cf. Lemmas \ref{l:M-L2} and \ref{l:M-L5} (iii)), it suffices to analyse the extensions to the left of $21211333^*11121211$, $12311333^*11121212$, i.e.,  $121211333^*11121211$, $221211333^*11121211$, $321211333^*11121211$, $212311333^*11121212$. 

\begin{lemma}\label{l:M-L15} If $131$ is forbidden, then $\lambda_0(321211333^*11121211)>\lambda_0(221211333^*11121211)\geq \lambda_0(221211333^*1112121132)>3.9387772$
\end{lemma} 

By the previous lemma, it suffices to analyse the extensions to the right of $121211333^*11121211$,  $212311333^*11121212$, i.e., $121211333^*111212111$, $121211333^*111212112$, $121211333^*111212113$, $212311333^*111212121$, $212311333^*111212122$, $212311333^*111212123$.

\begin{lemma}\phantomsection\label{l:M-L16}
\begin{itemize}
\item[(i)] $\lambda_0(121211333^*111212111)>\lambda_0(121211333^*111212112)>3.9387821$ 
\item[(ii)] if $312$ and $313$ are forbidden, then $\lambda_0(212311333^*11121212)\geq \lambda_0(212311333^*11121212311)>3.938776505$ 
\end{itemize}
\end{lemma}

By the previous lemma (and after recalling that $312$ is forbidden), it suffices to analyse the extensions to the left of  $121211333^*111212113$, i.e., $1121211333^*111212113$, $2121211333^*111212113$. 

\begin{lemma}\label{l:M-L17} $\lambda_0(2121211333^*111212113)< 3.93877609$
\end{lemma} 

By the previous lemma (and after recalling that $131$ is forbidden), it suffices to analyse the extensions to the right of $1121211333^*111212113$, i.e., $1121211333^*1112121132$, $1121211333^*1112121133$.

\begin{lemma}\label{l:M-L18} If $131$ and $211321$ are forbidden\footnote{Compare with Lemma \ref{l:M-L5} (i)}, then $\lambda_0(1121211333^*1112121132)\leq \lambda_0(231121211333^*11121211322)<3.938775922$
\end{lemma} 

By the previous lemma, we are led to investigate the extensions of $1121211333^*1112121133$. More concretely, the following statement is an immediate corollary of our discussions so far: 
\begin{corollary}\label{c:M-L1} Let $x\in\{1,2,3\}^{\mathbb{Z}}$ be a sequence such that $3.93877609 < m(x)=\lambda_0(x) < 3.938776505$. Then, 
$$\dots x_{-1} x_0^* x_1\dots = \dots 1121211333^*1112121133 \dots$$
\end{corollary}

\subsection{Self-replication} 

Our current goal is to describe the extensions of the string $1121211333^*1112121133$ leading to a Markov value strictly smaller than $3.93877624198\textcolor{red}{14}43$. 

For this sake, note that the extensions to the left of $1121211333^*1112121133$ are $11121211333^*1112121133$, $21121211333^*1112121133$, $31121211333^*1112121133$.

\begin{lemma}\label{l:M-L19} $\lambda_0(31121211333^*1112121133)>\lambda_0(21121211333^*1112121133)> 3.93877687$
\end{lemma}

By the previous lemma, it suffices to analyse the extensions to the right of $11121211333^*1112121133$, i.e., $11121211333^*11121211331$, $11121211333^*11121211332$, $11121211333^*11121211333$. 

\begin{lemma}\label{l:M-L20} $\lambda_0(11121211333^*11121211331)>\lambda_0(11121211333^*11121211332)> 3.938776301$ 
\end{lemma} 

By the previous lemma, it suffices to analyse the extensions to the left of $11121211333^*11121211333$, i.e., $111121211333^*11121211333$, $211121211333^*11121211333$, $311121211333^*11121211333$. 

\begin{lemma}\label{l:M-L21} $\lambda_0(111121211333^*11121211333)>\lambda_0(211121211333^*11121211333)> 3.938776282$  
\end{lemma}

By the previous lemma (and the fact that $312$ and $313$ are forbidden), it suffices to analyse the extensions to the right of $311121211333^*11121211333$, i.e., $311121211333^*1112121133311$, $311121211333^*111212113332$, $311121211333^*111212113333$. 

\begin{lemma}\label{l:M-L22} If $131$ is forbidden, then $\lambda_0(311121211333^*111212113333)> \lambda_0(311121211333^*111212113332)\geq \lambda_0(2311121211333^*111212113332)> 3.938776248$
\end{lemma} 

By the previous lemma (and after recalling that $131$, $22311$, $32311$, $123111$ are forbidden, cf Lemmas \ref{l:M-L1} (i), \ref{l:M-L3} (i), \ref{l:M-L4} (i)), it suffices to analyse the extensions to the left of $311121211333^*1112121133311$, i.e., $3311121211333^*1112121133311$. Now, we observe that the extensions to the left of $3311121211333^*1112121133311$ are $13311121211333^*1112121133311$, $23311121211333^*1112121133311$, $33311121211333^*1112121133311$. 

\begin{lemma}\label{l:M-L24} If $213$ and $3331113$ are forbidden, then $\lambda_0(13311121211333^*1112121133311)> \lambda_0(23311121211333^*1112121133311)\geq \lambda_0(\overline{21}23311121211333^*1112121133311\overline{12})=3.938776242699$
\end{lemma} 

By the previous lemma, it suffices to analyse the extensions to the right of $33311121211333^*1112121133311$, i.e., $33311121211333^*11121211333111$, $33311121211333^*11121211333112$, $33311121211333^*11121211333113$. 

\begin{lemma}\label{l:M-L25} $\lambda_0(33311121211333^*11121211333113)>\lambda_0(33311121211333^*11121211333112)> 3.93877624592$
\end{lemma} 

By the previous lemma (and after recalling that $213$ and $313$ are forbidden), it suffices to analyse the extensions to the left of $33311121211333^*11121211333111$, i.e., $1133311121211333^*11121211333111$, $233311121211333^*11121211333111$, $333311121211333^*11121211333111$. 

\begin{lemma}\label{l:M-L26} If $213$ and $3331113$ are forbidden, then $\lambda_0(333311121211333^*11121211333111)>\lambda_0(233311121211333^*11121211333111)\geq \lambda_0(233311121211333^*11121211333111\overline{21})>3.93877624206$
\end{lemma} 

By the previous lemma (and after recalling that $3331113$ is forbidden), it suffices to analyse the extensions to the right of $1133311121211333^*11121211333111$, i.e., $1133311121211333^*111212113331111$, $1133311121211333^*111212113331112$. 

\begin{lemma}\label{l:M-L27} $\lambda_0(1133311121211333^*111212113331111)>3.93877624309$
\end{lemma} 

By the previous lemma, it suffices to analyse the extensions to the right of $1133311121211333^*111212113331112$, i.e, \begin{itemize}
\item $1133311121211333^*1112121133311121$, 
\item $1133311121211333^*1112121133311122$, $1133311121211333^*1112121133311123$
\end{itemize}

\begin{lemma}\label{l:M-L28} $\lambda_0(1133311121211333^*1112121133311123)>\lambda_0(1133311121211333^*1112121133311122)> 3.938776242211$
\end{lemma} 

By the previous lemma (and after recalling that $213$ is forbidden), it suffices to analyse the extensions to the right of $1133311121211333^*1112121133311121$, i.e., $1133311121211333^*11121211333111211$, $1133311121211333^*11121211333111212$. 

\begin{lemma}\label{l:M-L29} $\lambda_0(1133311121211333^*11121211333111211)>3.93877624201$
\end{lemma} 

By the previous lemma (and after recalling that $3111333$, $2111333111212$, $11113331112121$ are forbidden, cf Lemmas \ref{l:M-L6} (i), \ref{l:M-L11}, \ref{l:M-L12} (i)), it suffices to analyse the extensions to the left of $1133311121211333^*11121211333111212$, i.e., $21133311121211333^*11121211333111212$, $31133311121211333^*11121211333111212$. As it turns out, the extensions to the right of these two words are: 

\begin{itemize}
\item $21133311121211333^*111212113331112121$, $31133311121211333^*111212113331112121$
\item $21133311121211333^*111212113331112122$, $31133311121211333^*111212113331112122$ 
\item $21133311121211333^*111212113331112123$, $31133311121211333^*111212113331112123$
\end{itemize} 

\begin{lemma}\label{l:M-L31} $\min\{\lambda_0(21133311121211333^*111212113331112123), \lambda_0(31133311121211333^*111212113331112123), \\ \lambda_0(31133311121211333^*111212113331112122)\}>\lambda_0(21133311121211333^*111212113331112122) \\
\geq \lambda_0(12121133311121211333^*111212113331112122) > 3.938776241990046,$
since $32113331112121$ and $22113331112121$ are forbidden by Lemma~\ref{l:M-L12}, $11211333111212$ is forbidden by Lemma~\ref{l:M-L13}, and $32121133311121211$ and $22121133311121211$ forbidden by Lemma~\ref{l:M-L15}.
\end{lemma} 

By the previous lemma (and after recalling that $213$ and $2121133311121212$ are forbidden, cf. Lemmas \ref{l:M-L2} and \ref{l:M-L14} (iv)), it suffices to analyse the extensions to the right of $21133311121211333^*111212113331112121$, $31133311121211333^*111212113331112121$, i.e., $21133311121211333^*1112121133311121211$, $31133311121211333^*1112121133311121211$. As it turns out, the extensions to the right of these two words are $21133311121211333^*11121211333111212113$, $31133311121211333^*11121211333111212113$ because the strings $121211333111212111$, $121211333111212112$ are forbidden (cf. Lemma \ref{l:M-L16} (i)). Finally, the resulting words extend to the right as
$$21133311121211333^{*}111212113331112121133$$
and
$$31133311121211333^{*}111212113331112121133$$
because $131$ and $11323$, $11322$, $211321$ are forbidden (cf. Lemmas~\ref{l:M-L3} (i) and~\ref{l:M-L5}
 (i)).

In summary, our discussion so far yields the following statement: 

\begin{corollary}\label{c:M-L2} Let $x\in\{1,2,3\}^{\mathbb{Z}}$ be a sequence with Markov value $m(x)<3.938776241990046$. If $x$ contains the string $1121211333^*1112121133$, say, 
$$x= \ldots x_{i-9} \ldots x_i^* \ldots x_{i+10}\ldots = \ldots 1121211333^*1112121133 \ldots,$$ 
then one has 
$$x=\ldots x_{i-15}\ldots x_i^*\ldots x_{i+21} = \ldots 1133311121211333^*11121211333^{**}1112121133\ldots$$ 
and the vicinity of $x_{i+11}^{**}$ is $1121211333^{**}1112121133$. In particular, by recursively analysing the positions $x_{i+11k}$, $k\in\mathbb{N}$, one actually has 
$$x=\ldots x_{i-15}\ldots x_i^*\ldots = \ldots 1133311121211333^*\overline{11121211333}$$
\end{corollary} 

Let
\[j_{0}:=\lambda_{0}(\overline{11121211333^{*}}) = 3.938776241981028026...\in L\]
and
\begin{eqnarray*}
j_{1}&:=&\lambda_{0}(\overline{21}233111331132123113331112121133311121211333^{*}111212113331112122\overline{32}) \\ &=& 3.93877624199054947868687...\in L
\end{eqnarray*} 

\begin{proposition}\label{p:M-L1}
If $j_{0} \leq m(a) = \lambda_{0}(a) < 3.9387762419922$ then (up to transposition) either
\begin{itemize}
\item $a = ...21133311121211333^{*}111212113331112122...$;
\item $a = ...21133311121211333^{*}\overline{11121211333}$; or
\item $a = ...31133311121211333^{*}\overline{11121211333}$.
\end{itemize} 
\end{proposition}

\begin{proof}
Since $j_0\leq m(a) = \lambda_{0}(a) < 3.9387762419922$, we can use Corollary~\ref{c:M-L1} and all of the results from Lemma \ref{l:M-L19} up to Lemma~\ref{l:M-L29}. Because 
\begin{align*}
&\min\{\lambda_0(21133311121211333^*111212113331112123), \lambda_0(31133311121211333^*111212113331112123), \\ & \lambda_0(31133311121211333^*111212113331112122)\}> 3.9387762419922,
\end{align*}
we can partly use Lemma~\ref{l:M-L31} together with the subsequent analysis to derive that either 
$$a=...21133311121211333^{*}111212113331112122...,\,\,a=...21133311121211333^*\overline{11121211333},$$
or
$$a=...31133311121211333^{*}\overline{11121211333}.$$
\end{proof}

\begin{proposition}\label{p:M-L2}
If $j_0<m(a) < 3.9387762419922$ and $a$ contains $$21133311121211333^{*}111212113331112122,$$ then $m(a) \geq j_{1}$.
\end{proposition}

\begin{proof}
As in Lemma~\ref{l:M-L31}, we are forced to have  
$$m(a) = \lambda_0(...12121133311121211333^{*}111212113331112122...).$$
Therefore, our task is reduced to check that if 
$$m(a) = \lambda_0(...12121133311121211333^{*}111212113331112122...),$$ 
then one actually has $m(a)\geq j_1$. For this sake, observe that 
\[\lambda_{0}(a)\geq\lambda_{0}(...112121133311121211333^{*}111212113331112122...).\] 
At this point, Lemmas~\ref{l:M-L19}, \ref{l:M-L21}, \ref{l:M-L24} and \ref{l:M-L26} force us to have
\[\lambda_{0}(a)\geq\lambda_{0}(...113331112121133311121211333^{*}111212113331112122...).\] 
Hence, 
\[\lambda_{0}(a)\geq\lambda_{0}(...123113331112121133311121211333^{*}111212113331112122...)\] 
since $131$, $32311$ and $22311$ are forbidden (cf. Lemmas~\ref{l:M-L1} and~\ref{l:M-L3}). It follows from Lemma~\ref{l:M-L5} (iii) that 
\[\lambda_{0}(a)\geq\lambda_{0}(...132123113331112121133311121211333^{*}111212113331112122...).\] 
After Lemmas~\ref{l:M-L2},~\ref{l:M-L4} (i),~\ref{l:M-L5} (i), one has  
\[\lambda_{0}(a)\geq\lambda_{0}(...3111331132123113331112121133311121211333^{*}111212113331112122...).\] 
By Lemmas~\ref{l:M-L1}(i),~\ref{l:M-L3} (i),~\ref{l:M-L4} (i),~\ref{l:M-L6} (i), the strings $131$, $23111$ and $3331113$ are forbidden, so that 
\[\lambda_{0}(a)\geq\lambda_{0}(\overline{21}233111331132123113331112121133311121211333^{*}111212113331112122...).\]

We also have that 
\[\lambda_{0}(a)\geq\lambda_{0}(\overline{21}233111331132123113331112121133311121211333^{*}1112121133311121223).\]

We claim that $a$ cannot contain $2231$. Indeed, Lemma~\ref{l:M-L2} forbids $22313$ and $22312$ since both contain $313$ or $312$, while Lemma~\ref{l:M-L3} forbids $22311$. So we see that $2231$ can never be extended.

We also claim that $a$ cannot contain $3231$. Indeed, Lemma~\ref{l:M-L2} forbids $32313$ and $32312$ since both contained $313$ or $312$, while Lemma~\ref{l:M-L3} forbids $32311$. So we see that $3231$ can never be extedned.

Therefore, since $2231$ is forbidden,
\[\lambda_{0}(a)\geq\lambda_{0}(\overline{21}233111331132123113331112121133311121211333^{*}11121211333111212232).\]
We also have that $3231$ is forbidden and so we find that
\[\lambda_{0}(a)\geq\lambda_{0}(\overline{21}233111331132123113331112121133311121211333^{*}111212113331112122\overline{32})=j_{1}.\]
\end{proof}

\begin{proposition}
The open interval $J=(j_{0},j_{1})$ is a maximal gap of $L$.
\end{proposition}

\begin{proof}
If $a$ is periodic and $j_{0} \leq m(a) \leq j_{1} < 3.9387762419922$, then Proposition~\ref{p:M-L1} tells us that $a = \overline{11121211333}$ in which case $m(a) = j_{0} \not\in J$, or $a$ contains $21133311121211333^{*}111212113331112122$. In the latter case, Proposition~\ref{p:M-L2} then tells us that $m(a) \geq j_{1}$ and so again $m(a) \not\in J$. Therefore, $J$ does not contain the Markov value of any periodic sequence and so, since the Lagrange spectrum is the closure of the set of Markov values of periodic sequences, we conclude that $J$ is indeed a maximal gap of $L$.
\end{proof}

\begin{proposition}\label{p:M-L3}
Let $a\in\{1,2,3\}^{\mathbb{Z}}$ be a sequence with Markov value 
$j_{0} < m(a) = \lambda_{0}(a) < j_{1}$
then $m_{1}\leq m(a) \leq m_{4}$, where
\[m_{1} = m(\overline{12}3311133113212121133311121211333^{*}\overline{11121211333})\]
\[=3.9387762419810960597...\]
and
\[m_{4} = m(\overline{12}331113311321231133311121211333^{*}\overline{11121211333})\]
\[=3.938776241989784909....\]
\end{proposition}

\begin{proof}
By Propositions~\ref{p:M-L1} and~\ref{p:M-L2}, we have that $a = ...21133311121211333^{*}\overline{11121211333}$ or $a = ...31133311121211333^{*}\overline{11121211333}$. 

We begin by analysing the former. Since $32113331112121$ and $22113331112121$ are forbidden by Lemma~\ref{l:M-L12}, $11211333111212$ is forbidden by Lemma~\ref{l:M-L13}, and $32121133311121211$ is forbidden by Lemma~\ref{l:M-L15}, we have 
$$a=\dots 12121133311121211333^*\overline{11121211333}.$$
Since $312$ is forbidden, this sequence extends to the left with $1$ or $2$. Suppose that it extends by a $1$. By Corollary \ref{c:M-L2}, and the same arguments we just made, we see that 
$$a= \dots 12121133311121211333^{***}11121211333^*\overline{11121211333}$$ 
and, once again, this word could extend on the left with $1$ or $2$. However, an extension with $2$ is not possible because this would force $\lambda_{-11}(a)>\lambda_0(a)=m(a)$, a contradiction. Continuing would leave us with $a = \overline{11121211333}$, so $m(a) = j_{0}$, which is also a contradiction.
So we must have
$$a= \dots 212121133311121211333^*\overline{11121211333}.$$ 

Now
\[m(a) \geq m(13212121133311121211333^{*}\overline{11121211333}).\]
By Lemma~\ref{l:M-L2}, $313$ and $213$ are forbidden in $a$ and so
\[m(a) \geq m(113212121133311121211333^{*}\overline{11121211333}).\]
Lemmas~\ref{l:M-L4} and~\ref{l:M-L5} forbid $111321$ and $2113212$, so we must have
\[m(a) \geq m(3113212121133311121211333^{*}\overline{11121211333}).\]
Similar arguments allow us to show that
\[m(a) \geq m(311133113212121133311121211333^{*}\overline{11121211333}).\]
Lemma~\ref{l:M-L1} forbids $131$. We claim that $23111$ is also forbidden. Lemma~\ref{l:M-L3} forbids $223111$ and $323111$ while Lemma~\ref{l:M-L4} forbids $123111$ and so $23111$ is never extendible and so must be forbidden. Therefore,
\[m(a) \geq m(3311133113212121133311121211333^{*}\overline{11121211333}).\]
Lemma~\ref{l:M-L6} prevents $3331113$ and so
\[m(a) \geq m(23311133113212121133311121211333^{*}\overline{11121211333}).\]
From here on, $312$ being forbidden by Lemma~\ref{l:M-L2} gives us that 
\[m(a) \geq m(\overline{12}3311133113212121133311121211333^{*}\overline{11121211333}) = m_{1}.\]


Now analysing the possibility that $a =  ...31133311121211333^{*}\overline{11121211333}$. 
%
Since $131$ is forbidden, we have
\[m(a) \leq m(231133311121211333^{*}\overline{11121211333}).\]
Now, we are forbidden to have $32311$ and $22311$ so we must have
\[m(a) \leq m(1231133311121211333^{*}\overline{11121211333}).\]
Next, since $1123113$ is forbidden, we must have
\[m(a) \leq m(21231133311121211333^{*}\overline{11121211333}).\]
Then
\[m(a) \leq m(321231133311121211333^{*}\overline{11121211333}).\]
Now we have
\[m(a) \leq m(1321231133311121211333^{*}\overline{11121211333}).\]
Since $313$ and $213$ are forbidden, we must have
\[m(a) \leq m(11321231133311121211333^{*}\overline{11121211333}).\]
Now $111321$ and $211321$ are forbidden so we must have
\[m(a) \leq m(311321231133311121211333^{*}\overline{11121211333}).\]
Then
\[m(a) \leq m(13311321231133311121211333^{*}\overline{11121211333}).\]
Since $313$ and $213$ are forbidden we get
\[m(a) \leq m(113311321231133311121211333^{*}\overline{11121211333}).\]
Then
\[m(a) \leq m(31113311321231133311121211333^{*}\overline{11121211333}).\]
Now $131$ is forbidden and extending by $2$ would lead to one of $32311$, $22311$, or $123111$ all of which are forbidden. So we obtain
\[m(a) \leq m(331113311321231133311121211333^{*}\overline{11121211333}).\]
We have that $3331113$ is forbidden and so we must have
\[m(a) \leq m(2331113311321231133311121211333^{*}\overline{11121211333}).\]
From here we obtain
\[m(a) \leq m(\overline{12}331113311321231133311121211333^{*}\overline{11121211333}) = m_{4}.\]
This completes the proof.
\end{proof}

An immediate consequence of our discussion so far is the following statement: 

\begin{corollary} $\HD((M\setminus L)\cap (j_0, j_1)) = \HD(K)$ where $K$ is the Gauss--Cantor set of continued fractions with entries $1$, $2$, $3$ not containing the following forbidden strings (nor their transposes): 
\begin{itemize}
\item $131$, $312$, $313$, $22311$, $32311$,  $123111$, $123112$, $1123113$, $3331113$, $2111333111212$, 
\item $11113331112121$, $11113331112122$, $22113331112121$, $32113331112121$, 
\item $111113331112123$, $112113331112121$, $211113331112123$, $3111133311121232$,
\item $3111133311121233$, $2121133311121212$, $331133311121212$, $22121133311121211$, 
\item $32121133311121211$, $121211333111212111$, $121211333111212112$, 
\item $21231133311121212$, $11212113331112121133$.
\end{itemize}
\end{corollary}

\begin{proof} Denote by $\mathcal{F}$ the set consisting of the strings above and their transposes. By Corollary \ref{c:M-L1}, if $x\in\{1,2,3\}^{\mathbb{Z}}$ and $j_0<m(x)<j_1$, then $\dots x_{-1} x_0^* x_1\dots = \dots 1121211333^*1112121133 \dots$ (up to transposition). Furthermore, the discussion before Corollary \ref{c:M-L1} says that $x$ doesn't contain the strings in $\mathcal{F}\setminus\{\gamma, \gamma^t\}$, where $\gamma=11212113331112121133$ is the ``self-replicating'' word and $\gamma^t$ is its transpose.  

By Propositions \ref{p:M-L1} and \ref{p:M-L2}, one actually has that 
$$x=y^t1133311121211333^{*}\overline{11121211333}$$ 
where $y\in\{1,2,3\}^{\mathbb{N}}$ doesn't contain strings from $\mathcal{F}\setminus\{\gamma,\gamma^t\}$. By Proposition \ref{p:M-L3} and Corollary \ref{c:M-L2}, either $y$ has the form $y=\delta\overline{11121211333}$ where $\delta$ is a \emph{finite} string or $y$ doesn't contain a string from $\mathcal{F}$. In particular, $M\cap (j_0,j_1)$ is included in the union of a countable set and a set which is bi-Lipschitz homeomorphic to $K$, so that $\HD((M\setminus L)\cap (j_0, j_1)) = \HD(M\cap(j_0,j_1)) \leq \HD(K)$. Since it is not hard to see that $(M\setminus L)\cap (j_0,j_1)$ contains the set 
$$\{m(y^t212121133311121211333^{*}\overline{11121211333}): y^t21212 \textrm{ doesn't contain strings from } \mathcal{F} \}$$ 
which is bi-Lipschitz homeomorphic to $K$, the argument is now complete. 
\end{proof} 

Performing calculations using the methods of Jenkinson-Pollicot~\cite{JP01}, we obtained heuristics suggesting that $0.593 < \HD(K') < \HD(K'') < 0.595$, where $K'$ is the Gauss--Cantor set of continued fractions with entries $1$, $2$, $3$ not containing the forbidden strings $131$, $312$, $313$, $22311$, $32311$,  $123111$, $123112$, $1123113$, $3331113$, and $11333111212$ (nor their transposes), and $K''$ is the Gauss--Cantor set of continued fractions with entries $1$, $2$, $3$ not containing the forbidden strings $131$, $312$, $313$, $22311$, $32311$,  $123111$, $123112$, $1123113$, $3331113$ (nor their transposes). Since the every forbidden string for $K$ has a subword that is a forbidden string for $K'$, we see that $K'\subset K$. Similarly, since the forbidden strings for $K''$ are a strict subset of those for $K$, we have $K\subset K''$. Hence we expect the heuristic
\[0.593 < \HD(K) < 0.595\]
to be true which would also give us that $\HD(M\setminus L) > 0.593$ - an improved lower bound.

\section{Freiman's gap}\label{s:Frei-gap}

In \cite[Section 10, pp.66--71]{Fr75}, G. Freiman proved the following result: 

\begin{theorem}\label{t:Freiman-last-gap} One has $M\cap(\nu,\mu)=\varnothing$ where 
$$\nu=[4;3,1,3,1,3,\overline{4,4,4,3,2,3}]+[0;3,1,3,1,2,1,1,3,3,\overline{3,1,3,1,2,1}]$$
and 
$$\mu=[4;4,3,2,2,\overline{3,1,3,1,2,1}]+[0;3,2,1,1,\overline{3,1,3,1,2,1}]$$
\end{theorem}

In this section, we extract key parts of the proof of this theorem. For this sake, we restrict from now on our attention to the sequences $\underline{a}=(a_n)_{n\in\mathbb{Z}}\in(\mathbb{N}^*)^{\mathbb{Z}}$ such that 
$$4<m(\underline{a})=\lambda_0(\underline{a})<5.$$
Note that these inequalities imply that 
$$\underline{a}\in\{1,2,3,4\}^{\mathbb{Z}} \quad \textrm{and} \quad a_0\in \{3,4\}.$$

\subsection{Preliminaries}

We require the following results the proofs of which can be found in~\cite[Appendix D]{L+20}. The first determine that the central portion of a candidate sequence giving rise to Markov values in the range $(\nu,\mu)$ must be (up to transposition) $...34^{*}3...$ or $...34^{*}4...$.

\begin{lemma}\label{l:Frei-gap2+3} If $m(\underline{a})<4.55$, then $\underline{a}\in\{1,2,3,4\}^{\mathbb{Z}}$ can not contain the subwords $41$, $42$ or their transposes.   
\end{lemma}

\begin{lemma}\label{l:Frei-gap4+5} If $m(\underline{a})<4.52786$, then $\underline{a}\in\{1,2,3,4\}^{\mathbb{Z}}$ can not contain the subwords $313133$, $443131344$ or their transposes.
\end{lemma}

\begin{corollary}\label{c:Frei-gap1} Suppose that $4.5278<m(\underline{a})=\lambda_0(\underline{a})<4.52786$. Then, $\underline{a}\in\{1,2,3,4\}^{\mathbb{Z}}$ has the form $\dots a_{-1}a_0a_1\dots= \dots 343\dots$ or $\dots344\dots$ (up to transposition). 
\end{corollary}

\subsection{Extensions of the word \boldmath{$343$}}

The following results analyse possible extensions of $...34^{*}3...$.

\begin{lemma}\label{l:Frei-gap6-9} If $m(\underline{a})<4.52786$, then $\underline{a}\in\{1,2,3,4\}^{\mathbb{Z}}$ can not contain the subwords $3432$, $134312$, $31343132$, $21313431312$ or their transposes.
\end{lemma}

\begin{corollary}\label{c:Frei-gap3} If $4.5278295<m(\underline{a})=\lambda_0(\underline{a})<4.5278296$ and $a_{-1}a_0a_1=343$, then $a_{-9}\dots a_0\dots a_{7}=33112131343131344$ (up to transposition). 
\end{corollary} 

\begin{lemma}\label{l:Frei-gap14+15} If $m(\underline{a})<4.528$, then $\underline{a}\in\{1,2,3,4\}^{\mathbb{Z}}$ can not contain the subwords $334$, $223444$ or their transposes. 
\end{lemma}

We include the proof of the following corollary as we will make use of the details in the next section.

\begin{corollary}\label{c:Frei-gap4} If $4.5278295<m(\underline{a})=\lambda_0(\underline{a})<4.5278296$ and $a_{-1}a_0a_1=343$, then $m(\underline{a})\leq \nu$. 
\end{corollary}

\begin{proof} By Corollary~\ref{c:Frei-gap3}, we have that $a_{-9}\dots a_0\dots a_7 = 33112131343131344$ (up to transposition). We want to \emph{maximize} $4.5278295<m(\underline{a})=\lambda_0(\underline{a})<4.5278296$. By Lemma~\ref{l:Frei-gap2+3}, this means that $a_{-9}\dots a_0\dots a_9= 3311213134313134443$. By Lemma~\ref{l:Frei-gap14+15}, we have $a_{-9}\dots a_0\dots a_{11}= 331121313431313444323$. By Lemma~\ref{l:Frei-gap6-9}, we derive $a_{-9}\dots a_0\dots a_{11}= 33112131343131344432344$. By repeating this argument, we conclude that $a_{-9}\dots a_0\dots a_7\dots= 33112131343131344\overline{432344}$. Similarly, we have from Lemma~\ref{l:Frei-gap14+15} that $a_{-10}\dots a_0\dots a_7 = 333112131343131344$. By Lemma~\ref{l:Frei-gap2+3}, we get $a_{-13}\dots a_0\dots a_7 = 131333112131343131344$. By Lemma~\ref{l:Frei-gap4+5}, $a_{-15}\dots a_0\dots a_7 = 12131333112131343131344$. By repeating this argument, we get $\dots a_{-9}\dots a_0\dots a_7 = \overline{121313}33112131343131344$. 

In summary, our assumptions imply the maximal value of $m(\underline{a})$ is $\nu$. 
\end{proof}

\subsection{Extensions of the word \boldmath{$344$}}

The following corollary results from an analysis of possible extensions of $...34^{*}4...$.

\begin{corollary}\label{c:Frei-gap7} If $4.5278291<m(\underline{a})=\lambda_0(\underline{a})<4.527832$ and $a_{-1}a_0a_1=344$, then $m(\underline{a})\geq \mu$.
\end{corollary}

\subsection{End of the proof of Theorem~\ref{t:Freiman-last-gap}} 

The desired result follows directly from Corollaries~\ref{c:Frei-gap1},~\ref{c:Frei-gap4} and~\ref{c:Frei-gap7}. 


\section{Gaps of the spectra nearby Freiman's gap}\label{s:accum-gaps}

In this section we prove Theorem~\ref{t:main2}. The proof of this theorem begins with the following lemmas. 

\begin{lemma}\label{l:accum-gaps1} If $4.5278295<m(\underline{a})=\lambda_0(\underline{a})<4.5278296$, then either $m(\underline{a})\geq \mu>\nu$ or $m(\underline{a})\leq\nu$ and, up to transposition,  
$$\underline{a} = \dots 3311213134^*3131344\dots$$
\end{lemma}

\begin{proof} This is a direct consequence of Corollaries~\ref{c:Frei-gap1},~\ref{c:Frei-gap3},~\ref{c:Frei-gap4} and~\ref{c:Frei-gap7}. 
\end{proof}

\begin{lemma}\label{l:accum-gaps2} The family of sets 
\begin{eqnarray*}
W_{n,m} &=& \{m(\underline{a})=\lambda_0(\underline{a})\in (4.5278295, \mu): \underline{a}=\underline{\theta}^t 323444313134^*313121133313121\underline{\theta}' \\ 
& & \textrm{ with }\underline{\theta}=\underbrace{444323\dots 444323}_{n \textrm{ times}}\widehat{\underline{\theta}}, \, \underline{\theta}'= \underbrace{313121\dots313121}_{m \textrm{ times}}\widetilde{\underline{\theta}}, \textrm{ and } \widehat{\underline{\theta}}, \widetilde{\underline{\theta}}\in\{1,2,3,4\}^{\mathbb{N}}\}  
\end{eqnarray*} 
indexed by $n,m\in\mathbb{N}$ is a basis of neighborhoods of $\nu$ in $M$. 
\end{lemma} 

\begin{proof} This follows directly from Lemma~\ref{l:accum-gaps1} and the proof of Corollary~\ref{c:Frei-gap4}. 
\end{proof}

\begin{lemma}\label{l:accum-gaps3} For each $n,m\in\mathbb{N}$, one has 
$$W_{n,m}\subset A_n+B_m$$ 
where 
$$A_n=\{[4;3,1,3,1,2,1,1,3,3,3,1,3,1,2,1+g^{n-1}(x)]: x\in K_1\},$$ 
$$B_m=\{[0;3,1,3,1,3,4,4,4,3,2,3+h^{m-1}(y)]: y\in K_2\},$$ 
the maps $g$, $h$ are $g(x)=[0;3,1,3,1,2,1+x]$, $h(y)=[0;4,4,4,3,2,3+y]$, and $K_1=\{[0;3,1,3,1,2,1,\widetilde{\underline{\theta}}]\in K\}$, $K_2=\{[0;4,4,4,3,2,3,\widehat{\underline{\theta}}]\in K\}$ with 
\begin{eqnarray*} 
K&=&\{[0;\underline{\theta}]:\underline{\theta}\in\{1,2,3,4\}^{\mathbb{N}} \textrm{ doesn't contain the strings } 14, 24, 433, 434, \\ & &131313, 2343, 223444, 123444 \textrm{ or their transposes}\}.
\end{eqnarray*}
\end{lemma}

\begin{proof} This is an immediate consequence of Lemma~\ref{l:accum-gaps2}, and the fact that Lemmas~\ref{l:Frei-gap2+3},~\ref{l:Frei-gap4+5},~\ref{l:Frei-gap6-9},~\ref{l:Frei-gap14+15} ensure that $\underline{a}\in\{1,2,3,4\}^{\mathbb{N}}$ with $m(\underline{a})<\mu$ can't contain the strings $14$, $24$, $433$, $434$, $131313, 2343, 223444, 123444$ or their transposes. 
\end{proof}

In view of Lemma~\ref{l:accum-gaps3}, our task is reduced to find gaps in the arithmetic sums $A_n+B_m$ for infinitely many pairs of indices $n$, $m$. In this direction, we observe that $K_1$ and $K_2$ are dynamical Cantor sets which are invariant under the contractions 
$$g(x)= [0;3,1,3,1,2,1+x] \quad \textrm{ and } \quad h(y)=[0;4,4,4,3,2,3+y]$$ 
whose fixed points are 
$$\alpha=[0;\overline{313121}] \quad \textrm{ and } \quad \beta = [0;\overline{444323}].$$ 
For subsequent reference, we note that $g$ and $h$ can be rewritten as 
$$g(x)= \frac{14x+19}{53x+72}, \quad h(y)=\frac{127x+436}{538x+1847}.$$
In particular, 
$$g'(x)=\frac{1}{(53x+72)^2}, \quad \quad |h'(y)| = \frac{1}{(538x+1847)^2}$$ 
and 
$$\alpha =\frac{2\sqrt{462}-29}{53}, \quad \quad \beta = \frac{\sqrt{243542}-430}{269}.$$

\begin{lemma}\label{l:accum-gaps4} One has $\alpha=\min K_1$, $\beta=\min K_2$, and 
$$\frac{\log|g'(\alpha)|}{\log|h'(\beta)|}\in\mathbb{R}\setminus\mathbb{Q}.$$
\end{lemma} 

\begin{proof} The fact that $\alpha=\min K_1$, $\beta=\min K_2$ follows from the definition of $K_1$, $K_2$ and the constraint on the continued fraction expansions of the elements of $K$. Furthermore, a straightforward computation yields 
$$g'(\alpha)= \frac{1}{(43+2\sqrt{462})^2} \quad \textrm{ and } \quad h'(\beta) = \frac{1}{(987+2\sqrt{243542})^2}.$$ 
Since $462=2\cdot 3\cdot 7\cdot 11$ and $243542=2\cdot 13\cdot 17\cdot 19\cdot 29$, their square roots generate distinct quadratic extensions of $\mathbb{Q}$ and 
$$g'(\alpha)^m=\frac{1}{(43+2\sqrt{462})^{2m}}\neq \frac{1}{(987+2\sqrt{243542})^{2n}}=h'(\beta)^n$$ 
for all $n, m\in\mathbb{N}^*$. Hence, $\frac{\log|g'(\alpha)|}{\log|h'(\beta)|}\in\mathbb{R}\setminus\mathbb{Q}$. This ends the proof of the lemma. 
\end{proof}

Also for later use, let us recall the following bound on the distortion of certain inverse branches of the Gauss map: 
\begin{lemma}\label{l:accum-gaps5} Let $f(x)=[0;a_1,\dots,a_k+x]$ be the inverse branch of the Gauss map associated to a finite word $(a_1,\dots,a_k)\in\{1,2,3,4\}^k$, $k\geq 1$. Then, 
$$\frac{1}{2.3}<\frac{|f'(x)|}{|f'(y)|}< 2.3$$ 
for any $\frac{\sqrt{2}-1}{2}\leq x,y\leq 2\sqrt{2}-2$.  
\end{lemma}

\begin{proof} Since $f(z)=\frac{p_{k-1}z+p_k}{q_{k-1} z+ q_k}$ and $|f'(z)|=\frac{1}{(q_{k-1} z+q_k)^2}$, where $\frac{p_j}{q_j}=[0;a_1,\dots,a_j]$ for all $1\leq j\leq k$, we have 
$$\frac{1}{2.3}< \left(\frac{1+\sqrt{2}}{2(2\sqrt{2}-1)}\right)^2 \leq \frac{|f'(x)|}{|f'(y)|} = \left(\frac{\frac{q_{k-1}}{q_k}y+1}{\frac{q_{k-1}}{q_k}x+1}\right)^2\leq \left(\frac{2(2\sqrt{2}-1)}{1+\sqrt{2}}\right)^2< 2.3$$ 
for $\frac{\sqrt{2}-1}{2}\leq x,y\leq 2\sqrt{2}-2$ (as $1/5\leq q_{k-1}/q_k\leq 1$). 
\end{proof}

An interesting consequence of this lemma is the fact that the sets $A_n$ and $B_m$ (cf. Lemma~\ref{l:accum-gaps3}) are mildly distorted ``copies'' of $K_1$ and $K_2$. For this reason, the next lemma about the ``thickness'' of $K_1$ and $K_2$ at their minima will be useful later. 

\begin{lemma}\label{l:accum-gaps6} Consider the intervals $R_0=[\alpha,\alpha_1]$, $U_0=(\alpha_1,\alpha_2)$, $L_0=[\beta,\beta_1]$ and $V_0=(\beta_1,\beta_2)$, where 
\begin{itemize}
\item $\alpha_1$ is the largest element of $K_1$ of the form $[0;3,1,3,1,2,1,3,\widetilde{\underline{\theta}}]$, 
\item $\alpha_2$ is the smallest element of $K_1$ of the form $[0;3,1,3,1,2,1,2,\widetilde{\underline{\theta}}]$, 
\item $\beta_1$ is the largest element of $K_2$ of the form $[0;4,4,4,3,2,3,4,\widehat{\underline{\theta}}]$,
\item $\beta_2$ is the smallest element of $K_2$ of the form $[0;4,4,4,3,2,3,3,\widehat{\underline{\theta}}]$.
\end{itemize}
Then, 
$$\frac{|R_0|}{|U_0|}<1 \quad \textrm{and} \quad \frac{|L_0|}{|V_0|}<\frac{1}{100}.$$
\end{lemma}

\begin{proof} Since the strings $41$, $42$ and $2343$ are forbidden in continued fraction expansions in $K$, we have that 
$\beta_1\leq [0;4,4,4,3,2,3,4,\overline{4,3}]$ and $\beta_2\geq [0;4,4,4,3,2,3,\overline{3,1}]$, and  
$$\frac{|L_0|}{|V_0|} = \frac{\beta_1-\beta}{\beta_2-\beta_1} < 0.008565 < \frac{1}{100}.$$ 
Similarly, we have $\alpha_1\leq [0;3,1,3,1,2,1,\overline{3,4}]$ and $\alpha_2\geq [0;3,1,3,1,2,1,2,\overline{1,3}]$, and 
$$\frac{|R_0|}{|U_0|} = \frac{\alpha_1-\alpha}{\alpha_2-\alpha_1} < 0.98479 < 1.$$ 
This completes the argument. 
\end{proof}

At this point, we are ready to complete the proof of Theorem~\ref{t:main2}. In fact, Lemmas~\ref{l:accum-gaps2} and~\ref{l:accum-gaps3} reduce our task to find gaps in $A_n+B_m$ for infinitely many $n, m\in\mathbb{N}^*$. Since 
$A_n=f_0\circ g^n(K_1)$ and $B_m=f_1\circ h^m(K_2)$, where 
$$f_0(x)= [4;3,1,3,1,2,1,1,3,3,3+x] \quad \textrm{ and } \quad f_1(x) = [0;3,1,3,1,3+x],$$  
and Lemma~\ref{l:accum-gaps4} ensures the denseness of $\{|g'(\alpha)|^n/|h'(\beta)|^m: n,m\in\mathbb{N}^*\}$ in $\mathbb{R}_+$, we get\footnote{Actually, using the general distortion bound statement in Chapter 4 of Palis--Takens book, it is possible to show that for any $c\in\mathbb{R}_+$ and $0<\varepsilon<1$, one has $c(1-\varepsilon)<\frac{|R_n|}{|L_m|} < c(1+\varepsilon)$ for infinitely many $n,m\in\mathbb{N}^*$.}, for any $c\in\mathbb{R}_+$, there are infinitely many $n,m\in\mathbb{N}^*$ such that 
$$\frac{c}{2}<\frac{|R_n|}{|L_m|} < 2c,$$ 
where $R_n=f_0\circ g^n(R_0)$ and $L_m=f_1\circ h^m(L_0)$. Because Lemma~\ref{l:accum-gaps5} also says that 
$$\frac{|L_m|}{|V_m|}<\frac{2.3}{100} \quad \textrm{ and } \quad \frac{|R_n|}{|U_n|}<2.3,$$ 
where $U_n=f_0\circ g^n(U_0)$, $V_m=f_1\circ h^m(V_0)$ are gaps of $A_n$ and $B_m$ (as $U_0$ and $V_0$ are gaps of $K_1$ and $K_2$), we conclude that 
$$\frac{|L_m|}{|U_n|} = \frac{|L_m|}{|R_n|}\cdot\frac{|R_n|}{|U_n|}<\frac{2}{c}\cdot 2.3 \quad \textrm{ and } \quad \frac{|R_n|}{|V_m|} = \frac{|R_n|}{|L_m|}\cdot\frac{|L_m|}{|V_m|}< 2c\cdot\frac{2.3}{100}.$$ 
Thus, if we take $c=5$, then 
$$\frac{|L_m|}{|U_n|} < 0.92 < 1 \quad \textrm{ and } \quad \frac{|R_n|}{|V_m|} < 0.23 < 1.$$ 
This ends the proof of Theorem~\ref{t:main2} because the inequalities above imply that $A_n+B_m$ has a gap: indeed, these estimates say that any parameter $t\in\mathbb{R}$ such that $t-U_n$ contains $L_m$ and their right endpoints are sufficiently close also satisfies $t-R_n\subset V_m$ and, \emph{a fortiori}, $(t-A_n)\cap B_m=\varnothing$ (see Figure~\ref{f:gaps}); hence, $A_n+B_m$ misses an entire open interval of parameters.

\begin{figure}
\begin{center}
\includegraphics[scale=1.2]{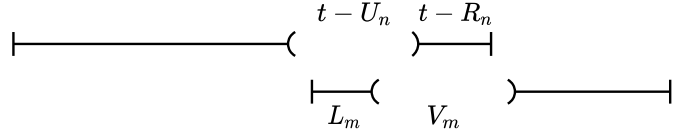}
\end{center}
\caption{Producing gaps in $A_{n}+B_{m}$.}
\label{f:gaps}
\end{figure}


\appendix
\section{Additional elements of $M\setminus L$}\label{app:M-L}

Here we present new elements of $M\setminus L$ that are less than those discussed in Section~\ref{s:M-L}. We only give the definitions of the sequences and Cantor sets involved and leave the proofs to the interested reader. These new sequences were also discovered using the computational search technique discussed in the introduction.

\subsection{Elements of {\boldmath $M\setminus L$} near to 3.676}

Computer investigations lead us to believe that there is a portion of $M\setminus L$ near to 3.676 given by an analysis of the subset of the real line near to
\[m(\overline{3^{*}21112123}) = 3.676699417246755742\ldots .\]

\subsection{Elements of {\boldmath $M\setminus L$} near to 3.726}

Computer investigations lead us to believe that there is a portion of $M\setminus L$ near to 3.726 given by an analysis of the subset of the real line near to
\[m(\overline{3322211121223^{*}}) = 3.726146224233042720\ldots .\]

Computer investigations also lead us to believe that there is a portion of $M\setminus L$ near to 3.726 given by an analysis of the subset of the real line near to
\[m(\overline{33222121223^{*}}) = 3.726278993734881116\ldots .\]

\subsection{Elements of {\boldmath $M\setminus L$} near to 3.942}

Computer investigations lead us to believe that there is a portion of $M\setminus L$ near to 3.942 given by an analysis of the subset of the real line near to
\[m(\overline{33211121232331113^{*}}) = 3.942001159911341469\ldots .\]
Note that this value is higher than the elements near to 3.938 that we rigorously considered in this paper. We chose not to analyse this sequence since, given its length, it would require a more involved analysis of the combinatorics without (in heuristic calculations) giving rise to an appreciable increase in the Hausdorff dimension estimates of $M\setminus L$.


\section{Pseudo-code for computer search}\label{app:alg}

Below is the pseudo-code for the part of the computer search that determines the central portion of sequences $\underline{a}\in\{1,2,3,4\}^{\Z}$ for which
\[m(\underline{a}) = \lambda_{0}(\underline{a}) \in [l,n],\]
for some interval $[l,n]$.

The code can also be used to `confirm' results about gaps in the spectra. For example, when running the code on intervals like $(0,\sqrt{5})$, $(\sqrt{12},\sqrt{13})$ or other known gaps the code terminates and returns an empty list of candidate sequences. On closed intervals, if the endpoints correspond to unique sequences, the code will return a two element list of finite sequences approaching the sequences corresponding to the endpoints.

\begin{algorithm}[H]\label{al:search}
\caption{- Find sequences whose Markov values could lie in the range $[l,n]$}
\begin{algorithmic}
\State $candidates \gets [1^{*},2^{*},3^{*},4^{*}]$
\State $forbidden\_words \gets []$
\State $alphabet \gets \{\_,1,2,3,4\}$ \hfill \emph{\# $\_$ is the empty string}
\State $extensions \gets (alphabet\times alphabet)\setminus\{(\_,\_)\}$
\State $l \gets l$
\State $n \gets n$
\State $length\_limit \gets$ maximum length of sequences to search up to
\State $min\_seq\_len \gets$ minimum length of all sequences in $candidates$
\While{$min\_seq\_len < length\_limit$ \textbf{and} $candidates \neq []$}
	\For{$sequence$ \textbf{in} $candidates$}
		\State $allowable \gets True$
		\State remove $sequence$ from $candidates$
		\For{$(x,y)$ \textbf{in} $extensions$}
			\State $trial\_sequence \gets \text{concatenation}(x,sequence,y)$
			\If{$trial\_sequence$ contains any words from $forbidden\_words$}
				\State \textbf{continue} \hfill \emph{\# the sequence is forbidden so move on to the next}
			\EndIf
			\State $\lambda_{max} \gets$ maximum possible value of $\lambda_{0}(trial\_sequence)$
			\If{$\lambda_{max} < l$}
				\State \textbf{continue} \hfill \emph{\# $\lambda_{0}$ is too small so move on to the next sequence}
			\EndIf
			\For{$z$ \textbf{in} $trial\_sequence$}
				\State $j \gets$ position of $z$ in $trial\_sequence$
				\State $\lambda_{min} \gets$ minimum possible value of $\lambda_{j}(trial\_sequence)$
				\If{$\lambda_{min} > n$}
					\State append $trial\_sequence$ to $forbidden\_words$
					\State $allowable \gets False$ \hfill \emph{\# the Markov value is too large}
				\EndIf
			\EndFor
			\If{$allowable$} \hfill \emph{\# the Markov value can lie in $[l,n]$}
				\State append $trial\_sequence$ to $candidates$
			\EndIf
		\EndFor
	\EndFor
	\If{$candidates \neq []$}
		\State $min\_seq\_len \gets$ minimum length of all sequences in $candidates$
	\EndIf
\EndWhile
\State \Return $candidates$
\end{algorithmic}
\end{algorithm}


\end{document}